\documentclass[10pt,a4paper]{article}

\usepackage{graphicx}
\usepackage{color}
\usepackage{indentfirst}
\usepackage{amsmath,amssymb,bm}
\usepackage{cases}
\usepackage{latexsym,bm,amsmath,amssymb,amsthm}

\begin{document}

\title{\bfseries  \Large Quasi-Einstein and Generalized Quasi-Einstein Warped Products with an Affine Connection}
\author{\normalsize Quan Qu \thanks{Corresponding author.
\newline \mbox{} \quad\,   E-mail addresses: quq453@nenu.edu.cn.
\newline  \mbox{} \quad\,   School of Mathematics and Statistics, Northeast Normal University, Changchun Jilin, 130024, China}}\date{}
\maketitle

\begin{abstract}
In this paper, we study the quasi-Einstein and generalized quasi-Einstein warped products with a semi-symmetric non-metric connection. We give the expressions of the Ricci tensors and scalar curvatures for the bases and fibres. In some cases we give some obstructions to the existence of the quasi-Einstein and generalized quasi-Einstein warped products with a semi-symmetric non-metric connection.
\end{abstract}
\paragraph{} \mbox{}\quad \, \normalsize 2010 Mathematics Subject Classification. 53B05, 53C25.
\section{\large Introduction}
The warped product $B \times _{f}F$ of two Riemannian manifolds $(B,g_{B})$ and $(F,g_{F})$ with a smooth function $f : B \to(0,\infty)$ is a product manifold of form $B \times F$ with the metric tensor $g = g_{B}\oplus f^2g_{F}$. Here, $(B,g_{B})$ is called the base manifold and $(F,g_{F})$ is called as the fiber manifold and $f$ is called as the warping function. The concept of warped products was first introduced by Bishop and O'Neill in [1] to construct examples of Riemannian manifolds with negative curvature. In [2], F. Dobarro and E. Dozo had studied the problem of showing when a Riemannian metric of constant scalar curvature can be produced on a product manifolds by a warped product construction. In [3], M.C. Chaki and R.K. Maity introduce the notion of quasi-Einstein manifold. Then the notion was generalized to generalized quasi-Einstein manifold in [4]. Later in [5], Dan Dumitru give some obstructions to the existence of quasi-Einstein warped products.

The definition of a semi-symmetric metric connection was given by H. Hayden in [6]. In 1970, K. Yano in [7] considered a semi-symmetric metric connection and studied some of its properties. In [8] and 9], Agashe and Chafle introduced the notion of a semi-symmetric non-metric connection and studied some of its properties and submanifolds of a Riemannian
manifold with semi-symmetric non-metric connections. Then in [10], S. Sular and C. \"Ozg\"ur studied the curvature of warped products with a semi-symmetric non-metric connection.

This paper is arranged as follows. In Section 2, we give the curvature of warped products with a semi-symmetric non-metric connection. In Section 3, we compute Ricci curvature and scalar curvature of quasi-Einstein warped products with a semi-symmetric non-metric connection, then we prove some obstructions to the existence of quasi-Einstein warped products with a semi-symmetric non-metric connection. In Section 4, we compute Ricci curvature and scalar curvature of generalized quasi-Einstein warped products with a semi-symmetric non-metric connection, then we prove some obstructions to the existence of generalized quasi-Einstein warped products with a semi-symmetric non-metric connection.

\section{\large Preliminaries}\label{sec2}
Let $ M $ be a Riemannian manifold with Riemannian metric $g$. A linear connection $\overline\nabla$ on a Riemannian manifold $M$ is called a semi-symmetric connection if the torsion tensor $T$ of the connection $\overline\nabla$
$$T(X, Y ) = \overline\nabla_{X}Y-\overline\nabla_{Y}X-[X,Y] \eqno{(1)}$$
satisfies  $$T(X, Y ) = \pi(Y)X-\pi(X)Y  \eqno{(2)}$$
where $\pi$ is a 1-form associated with the vector field $P$ on $M$ defined by $\pi(X)=g(X,P)$. $\overline\nabla$ is called a semi-symmetric metric connection if it satisfies $\overline\nabla g=0$. $\overline\nabla$ is called a semi-symmetric non-metric connection if it satisfies $\overline\nabla g\neq0$.

If $\nabla$ is the Levi-Civita connection of $M$, the semi-symmetric non-metric connection $\overline\nabla$ is given by
$$\overline\nabla_{X}Y=\nabla_{X}Y+\pi(Y)X. \eqno{(3)}$$
(see [8]). Let $R$ and $\overline R$ be the curvature tensors of $\nabla$ and $\overline\nabla$, respectively. Then $R$ and $\overline R$ are related by
$$\overline{R}(X,Y)Z=R(X,Y)Z+g(Z,\nabla_{X}P)Y-g(Z,\nabla_{Y}P)X+\pi(Z)[\pi(Y)X-\pi(X)Y]   \eqno{(4)}$$
for any vector fields $X,Y,Z$ on $M$, where the curvature $R$ is defined by $R(X,Y)=\nabla_{X}\nabla_{Y}-\nabla_{Y}\nabla_{X}-\nabla_{[X,Y]}.$ \par
We define the Ricci curvature and scalar curvature as follows:
$$Ric(X,Y)=\sum_{k}\varepsilon_{k}\langle R(X,E_{k})Y,E_{k}\rangle,$$
$$S=\sum_{k}\varepsilon_{k}Ric(E_{k},E_{k}),$$
where $E_{k}$ is an orthonormal base of $M$ with $\langle E_{k},E_{k}\rangle=\varepsilon_{k},\;\varepsilon_{k}=\pm1.$
\\The Hessian of $f$ is defined by $H^{f}(X,Y)=XYf-(\nabla_{X}Y)f,$\\
The Laplacian of $f$ is defined by $\Delta f=-Tr(H^{f}),$\\
The divergence of $P$ is defined by $divP=\sum_{k}\varepsilon_{k}\langle \nabla_{E_{k}}P,E_{k}\rangle.$\par

\newtheorem{Definition}{Definition}[section]
\begin{Definition}
Let $(B,g_{B})$ and $(F,g_{F})$ be two Riemannian manifolds, ${\rm dim}B=n_{1}$ and ${\rm dim}F=n_{2},\;f:B \to(0,\infty)$ be a smooth function. The {\bfseries warped product} $M=B\times _{f}F$ is the product manifold $B \times F$ with the metric tensor $g = g_{B}\oplus f^2g_{F}$. The function $f$ is called the warping function of the warped product. The warped product is called a {\bfseries simply Riemannian product} if $f$ is a constant function.
\end{Definition}

\begin{Definition}
A non-flat Riemannian manifold $(M, g)$ is said to be a {\bfseries quasi-Einstein} manifold if its Ricci tensor
$Ric^{M}$ satisfies the condition $Ric^{M}(X,Y)=ag(X,Y)+bA(X)A(Y)$ for every $X,Y\in\Gamma(TM),$ where $a,b$ are real scalars and $A$ is a non-zero 1-form on $M$ such that $A(X)=g(X,U)$ for all vector field $X\in\Gamma(TM),\;U$ being an unit vector field which is called the generator of the manifold. If $b=0$ then the manifold reduces to an Einstein space.
\end{Definition}

\begin{Definition}
A non-flat Riemannian manifold $(M, g)$ is said to be a {\bfseries generalized quasi-Einstein} manifold if its Ricci tensor $Ric^{M}$ is non-zero and satisfies the condition $Ric^{M}(X,Y)=ag(X,Y)+bA(X)A(Y)+cB(X)B(Y)$ for every $X,Y\in\Gamma(TM),$ where $a,b,c$ are real scalars and $A,B$ are two non-zero 1-form on $M$. The unit vector fields $U_{1}$ and $U_{2}$ corresponding to the 1-forms $A$ and respectively $B$ are defined by $A(X)=g(X,U_{1}),\;B(X)=g(X,U_{2})$,
and are orthogonal, i.e. $g(U_{1},U_{2})=0.$ If $c=0$ then the manifold reduces to a quasi-Einstein manifold.
\end{Definition}

We need the following two lemmas from [10]:
\newtheorem{Lemma}[Definition]{Lemma}
\begin{Lemma}
Let $M=B\times _{f}F$ be a warped product with metric $g=g_{B}\oplus f^2g_{F}$. If $X,Y,Z\in\Gamma(TB),\;U,V,W\in\Gamma(TF)$ and $P\in\Gamma(TB),$ then\\
$(1)\overline R(X,Y)Z=\overline R^B(X,Y)Z; $  \\
$(2)\overline R(V,X)Y=-\frac{H^f_{B}(X,Y)}{f}V-g(Y,\nabla_{X}P)V+\pi(X)\pi(Y)V;$   \\
$(3)\overline R(X,Y)V=\overline R(V,W)X=0; $   \\
$(4)\overline R(X,V)W=-g(V,W)\Big[\frac{\nabla_{X}^B(grad_{B}f)}{f}+\frac{Pf}{f}X\Big];$  \\
$(5)\overline R(U,V)W=R^F(U,V)W-\Big[\frac{|grad_{B}f|^2_{B}}{f^2}+\frac{Pf}{f}\Big]\Big[g(V,W)U-g(U,W)V\Big]. $
\end{Lemma}
\begin{Lemma}
Let $M=B\times _{f}F$ be a warped product with metric $g=g_{B}\oplus f^2g_{F}$. If $X,Y,Z\in\Gamma(TB),\;U,V,W\in\Gamma(TF)$ and $P\in\Gamma(TF),$ then\\
$(1)\overline R(X,Y)Z=R^B(X,Y)Z; $  \\
$(2)\overline R(V,X)Y=-\frac{H^f_{B}(X,Y)}{f}V-\pi(V)\frac{Yf}{f}X;$   \\
$(3)\overline R(X,Y)V=\pi(V)\Big[\frac{Xf}{f}Y-\frac{Yf}{f}X\Big]; $   \\
$(4)\overline R(V,W)X=\frac{Xf}{f}[\pi(W)V-\pi(V)W]; $   \\
$(5)\overline R(X,V)W=-g(V,W)\frac{\nabla_{X}^B(grad_{B}f)}{f}+\frac{Xf}{f}\pi(W)V-g(W,\nabla_{V}P)X+\pi(W)\pi(V)X;$ \\
$(6)\overline R(U,V)W=R^F(U,V)W-\frac{|grad_{B}f|^2_{B}}{f^2}[g(V,W)U-g(U,W)V]+g(W,\nabla_{U}P)V
-g(W,\nabla_{V}P)U\\ \mbox{}\quad\quad\quad\quad\quad\quad\quad+\pi(W)[\pi(V)U-\pi(U)V]. $
\end{Lemma}
Let $\overline{Ric}$ be the Ricci curvature of $\overline{\nabla},$ then by Lemma 2.4, 2.5 and the definition of the Ricci curvature, we can get:
\begin{Lemma}
Let $M=B \times_{f}F $ be a warped product, ${\rm dim}B=n_{1}, {\rm dim}F=n_{2}, \mbox{ and }{\rm dim}M=\overline n=n_{1}+n_{2}$. If $X,Y\in\Gamma(TB)$, $V,W\in\Gamma(TF)$ and $P\in\Gamma(TB),$ then: \\
$(1)\overline {Ric}(X,Y)=\overline {Ric}^B(X,Y)+n_{2}\Big[\frac{H^f_{B}(X,Y)}{f}+g(Y,\nabla_{X}P)-\pi(X)\pi(Y)\Big]; $\\
$(2)\overline {Ric}(X,V)=\overline {Ric}(V,X)=0; $   \\
$(3)\overline {Ric}(V,W)=\overline {Ric}^F(V,W)+\Big\{-\frac{\Delta_{B}f}{f}+(n_{2}-1)\frac{|grad_{B}f|^2_{B}}{f^2}
+(\overline n-1)\frac{Pf}{f}\Big\}g(V,W). $\\
\end{Lemma}
\begin{Lemma}
Let $M=B \times_{f}F $ be a warped product, ${\rm dim}B=n_{1}, {\rm dim}F=n_{2}, \mbox{ and } {\rm dim}M=\overline n=n_{1}+n_{2}$. If $X,Y\in\Gamma(TB)$, $V,W\in\Gamma(TF)$ and $P\in\Gamma(TF),$ then: \\
$(1)\overline {Ric}(X,Y)={Ric}^B(X,Y)+n_{2}\frac{H^f_{B}(X,Y)}{f}; $ \\
$(2)\overline {Ric}(X,V)=(\overline n-1)\pi(V)\frac{Xf}{f}; $   \\
$(3)\overline {Ric}(V,X)=(1-\overline n)\pi(V)\frac{Xf}{f}; $   \\
$(4)\overline {Ric}(V,W)=Ric^F(V,W)+g(V,W)\Big[-\frac{\Delta_{B}f}{f}+(n_{2}-1)\frac{|grad_{B}f|^2_{B}}{f^2}\Big]
+(\overline n-1)g(W,\nabla_{V}P)\\ \mbox{} \quad\quad\quad\quad\quad\quad\;\;+(1-\overline n)\pi(V)\pi(W). $
\end{Lemma}
Let $\overline{S}$ be the scalar curvature of $\overline{\nabla},$ then by lemma 2.6, 2.7 and the definition of the scalar curvature, we have the following:
\begin{Lemma}
Let $M=B \times_{f}F $ be a warped product, ${\rm dim}B=n_{1}, {\rm dim}F=n_{2}, \mbox{ and }{\rm dim}M=\overline n=n_{1}+n_{2}$. If $P\in\Gamma(TB),$ then the scalar curvature $\overline S$ has the following expression:
$$\overline S=\overline S^B-2n_{2}\frac{\Delta_{B}f}{f}+\frac{S^F}{f^2}+n_{2}(n_{2}-1)\frac{|grad_{B}f|^2_{B}}{f^2}
+n_{2}(\overline n-1)\frac{Pf}{f}-n_{2}\pi(P)+n_{2}div_{B}P. \eqno{(5)}$$
\end{Lemma}
\begin{Lemma}
Let $M=B \times_{f}F $ be a warped product, ${\rm dim}B=n_{1}, {\rm dim}F=n_{2}, \mbox{ and }{\rm dim}M=\overline n=n_{1}+n_{2}$. If $P\in\Gamma(TF),$ then the scalar curvature $\overline S$ has the following expression:
$$\overline S=S^B-2n_{2}\frac{\Delta_{B}f}{f}+\frac{S^F}{f^2}+n_{2}(n_{2}-1)\frac{|grad_{B}f|^2_{B}}{f^2}
+(1-\overline n)\pi(P)+(\overline n-1)div_{F}P. \eqno{(6)}$$
\end{Lemma}
\section{\large Quasi-Einstein warped products with a semi-symmetric non-metric connection}
\subsection{\large Ricci curvature and scalar curvature}
In this section, we compute Ricci curvature and scalar curvature of quasi-Einstein warped products with a semi-symmetric non-metric connection.

\newtheorem{Proposition}{Proposition}[section]
\begin{Proposition}
Let $M=B \times_{f}F$ be a quasi-Einstein warped product, ${\rm dim}B=n_{1},\;{\rm dim}F=n_{2},\;{\rm dim}M=\overline n=n_{1}+n_{2},\;X,Y\in\Gamma(TB),\;V,W\in\Gamma(TF),$ and $P\in\Gamma(TB).$\\
$1)$ If $U\in\Gamma(TB),$ then
\begin{equation}
\begin{cases}\setcounter{equation}{7}
\overline{Ric}^{B}(X,Y)=ag_{B}(X,Y)-n_{2}\Big[\frac{H^f_{B}(X,Y)}{f}+g(Y,\nabla_{X}P)-\pi(X)\pi(Y)\Big]
+bg_{B}(X,U)g_{B}(Y,U),\\
\overline{Ric}^{F}(V,W)=g_{F}(V,W)\big[f\Delta_{B}f+(1-n_{2})|grad_{B}f|_{B}^{2}+(1-\overline{n})fPf+af^{2}\big].
\end{cases}
\end{equation}
$2)$ If $U\in\Gamma(TF),$ then
\begin{equation}
\begin{cases}\setcounter{equation}{8}
\overline{Ric}^{B}(X,Y)=ag_{B}(X,Y)-n_{2}\Big[\frac{H^f_{B}(X,Y)}{f}+g(Y,\nabla_{X}P)-\pi(X)\pi(Y)\Big],\\
\overline{Ric}^{F}(V,W)=g_{F}(V,W)\big[f\Delta_{B}f+(1-n_{2})|grad_{B}f|_{B}^{2}+(1-\overline{n})fPf+af^{2}\big]\\
\mbox{}\quad\quad\quad\quad\quad\quad+bf^{4}g_{F}(V,U)g_{F}(W,U).
\end{cases}
\end{equation}
\end{Proposition}
\begin{proof}
$1)$ Since $M$ is quasi-Einstein, we have $\overline{Ric}(X,Y)=ag_{B}(X,Y)+bg_{B}(X,U)g_{B}(Y,U),$ then use Lemma 2.6(1) we get the first equation of $(7).$\\
Since $M$ is quasi-Einstein, we have $\overline{Ric}(V,W)=af^{2}g_{F}(V,W)+bg(V,U)g(W,U)=af^{2}g_{F}(V,W),$ then use Lemma 2.6(3) we get the second equation of $(7).$\\
$2)$ Since $M$ is quasi-Einstein, we have $\overline{Ric}(X,Y)=ag_{B}(X,Y)+bg(X,U)g(Y,U)=ag_{B}(X,Y),$ then use Lemma 2.6(1) we get the first equation of $(8).$\\
Since $M$ is quasi-Einstein, we have $\overline{Ric}(V,W)=af^{2}g_{F}(V,W)+bf^{4}g_{F}(V,U)g_{F}(W,U),$ then use Lemma 2.6(3) we get the second equation of $(8).$
\end{proof}
Using the same method, we can get the following Proposition:
\begin{Proposition}
Let $M=B \times_{f}F$ be a quasi-Einstein warped product, ${\rm dim}B=n_{1},\;{\rm dim}F=n_{2},\;{\rm dim}M=\overline n=n_{1}+n_{2},\;X,Y\in\Gamma(TB),\;V,W\in\Gamma(TF),$ and $P\in\Gamma(TF).$\\
$1)$ If $U\in\Gamma(TB),$ then
\begin{equation}
\begin{cases}\setcounter{equation}{9}
Ric^{B}(X,Y)=ag_{B}(X,Y)-n_{2}\frac{H^f_{B}(X,Y)}{f}+bg_{B}(X,U)g_{B}(Y,U),  \\
Ric^{F}(V,W)=g_{F}(V,W)\big[f\Delta_{B}f+(1-n_{2})|grad_{B}f|_{B}^{2}+af^{2}\big]
+(1-\overline n)g(W,\nabla_{V}P)\\ \mbox{}\quad\quad\quad\quad\quad\quad+(\overline n-1)\pi(V)\pi(W).
\end{cases}
\end{equation}
$2)$ If $U\in\Gamma(TF),$ then
\begin{equation}
\begin{cases}\setcounter{equation}{10}
Ric^{B}(X,Y)=ag_{B}(X,Y)-n_{2}\frac{H^f_{B}(X,Y)}{f},   \\
Ric^{F}(V,W)=g_{F}(V,W)\big[f\Delta_{B}f+(1-n_{2})|grad_{B}f|_{B}^{2}+af^{2}\big]+(1-\overline n)g(W,\nabla_{V}P)\\ \mbox{}\quad\quad\quad\quad\quad\quad+(\overline n-1)\pi(V)\pi(W)+bf^{4}g_{F}(V,U)g_{F}(W,U).
\end{cases}
\end{equation}
\end{Proposition}
By Lemma 2.8, Proposition 3.1 and the definition of scalar curvature, we have:
\begin{Proposition}
Let $M=B \times_{f}F$ be a quasi-Einstein warped product, ${\rm dim}B=n_{1},\;{\rm dim}F=n_{2},\;{\rm dim}M=\overline n=n_{1}+n_{2},\;X,Y\in\Gamma(TB),\;V,W\in\Gamma(TF),$ and $P\in\Gamma(TB).$\\
$1)$ If $U\in\Gamma(TB),$ then
\begin{equation}
\begin{cases}
\overline{S}^{M}=\overline{n}a+b;    \\
\overline{S}^{B}=n_{1}a+n_{2}\frac{\Delta_{B}f}{f}-n_{2}div_{B}P+n_{2}\pi(P)+b;\\
\overline{S}^{F}=n_{2}\big[f\Delta_{B}f+(1-n_{2})|grad_{B}f|_{B}^{2}+(1-\overline{n})fPf+af^{2}\big].
\end{cases}
\end{equation}
$2)$ If $U\in\Gamma(TF),$ then
\begin{equation}
\begin{cases}
\overline{S}^{M}=\overline{n}a+b;    \\
\overline{S}^{B}=n_{1}a+n_{2}\frac{\Delta_{B}f}{f}-n_{2}div_{B}P+n_{2}\pi(P);\\
\overline{S}^{F}=n_{2}\big[f\Delta_{B}f+(1-n_{2})|grad_{B}f|_{B}^{2}+(1-\overline{n})fPf+af^{2}\big]+bf^{2}.
\end{cases}
\end{equation}
\end{Proposition}
Similarly we have:
\begin{Proposition}
Let $M=B \times_{f}F$ be a quasi-Einstein warped product, ${\rm dim}B=n_{1},\;{\rm dim}F=n_{2},\;{\rm dim}M=\overline n=n_{1}+n_{2},\;X,Y\in\Gamma(TB),\;V,W\in\Gamma(TF),$ and $P\in\Gamma(TF).$\\
$1)$ If $U\in\Gamma(TB),$ then
\begin{equation}
\begin{cases}
\overline{S}^{M}=\overline{n}a+b;    \\
S^{B}=n_{1}a+n_{2}\frac{\Delta_{B}f}{f}+b;\\
S^{F}=n_{2}\big[f\Delta_{B}f+(1-n_{2})|grad_{B}f|_{B}^{2}+af^{2}\big]+(1-\overline{n})div_{F}P
+(\overline{n}-1)\pi(P).
\end{cases}
\end{equation}
$2)$ If $U\in\Gamma(TF),$ then
\begin{equation}
\begin{cases}
\overline{S}^{M}=\overline{n}a+b;    \\
S^{B}=n_{1}a+n_{2}\frac{\Delta_{B}f}{f};\\
S^{F}=n_{2}\big[f\Delta_{B}f+(1-n_{2})|grad_{B}f|_{B}^{2}+af^{2}\big]+(1-\overline{n})div_{F}P
+(\overline{n}-1)\pi(P)+bf^{2}.
\end{cases}
\end{equation}
\end{Proposition}

\subsection{\large Obstructions to the existence of quasi-Einstein warped products with a semi-symmetric non-metric connection}
In this section, we prove some obstructions to the existence of quasi-Einstein warped products with a semi-symmetric non-metric connection. We consider the following four cases:
\subsubsection{\normalsize{\bfseries When \boldmath $P\in\Gamma(TB),\;U\in\Gamma(TB).$}}
In the second equation of $(7),$ let:
$$f\Delta_{B}f+(1-n_{2})|grad_{B}f|_{B}^{2}+(1-\overline{n})fPf+af^{2}\triangleq \alpha_{1}.   \eqno{(15)} $$
Then the equation $(7)$ becomes:
\begin{equation}
\begin{cases}\setcounter{equation}{16}
\overline{{Ric}}^{B}(X,Y)=ag_{B}(X,Y)-n_{2}\Big[\frac{H^f_{B}(X,Y)}{f}+g(Y,\nabla_{X}P)-\pi(X)\pi(Y)\Big]
+bg_{B}(X,U)g_{B}(Y,U),   \\
\overline{Ric}^{F}(V,W)=\alpha_{1} g_{F}(V,W),  \\
f\Delta_{B}f+(1-n_{2})|grad_{B}f|_{B}^{2}+(1-\overline{n})fPf+af^{2}=\alpha_{1}.
\end{cases}
\end{equation}
\newtheorem{Theorem}[Proposition]{Theorem}
\begin{Theorem}
Let $M=B \times_{f}F$ be a quasi-Einstein warped product with $B$ compact and connected, ${\rm dim}B=n_{1}\geqslant1,\;{\rm dim}F=n_{2}\geqslant2,\;{\rm dim}M=\overline n=n_{1}+n_{2},\;P\in\Gamma(TB)$ and $U\in\Gamma(TB).$
If $n_{1}=1,\;div_{B}P=c_{1}$ and $\pi(P)=c_{2}$ are both constants, then $M$ is a sample Riemannian product.
\end{Theorem}
\begin{proof}
If $n_{1}=1,$ then $\overline{S}^{B}=0,$ by the second equation of $(11)$ and $div_{B}P=c_{1},\;\pi(P)=c_{2}$
we can get
$$0=a+n_{2}\frac{\Delta_{B}f}{f}-n_{2}c_{1}+n_{2}c_{2}+b\Rightarrow \frac{\Delta_{B}f}{f}=c_{1}-c_{2}-\frac{a+b}{n_{2}}\triangleq c_{0}\Rightarrow \Delta_{B}f=c_{0}f.$$
Then the Laplacian has constant sign and hence $f$ is constant.
\end{proof}

\begin{Theorem}
Let $M=B \times_{f}F$ be a quasi-Einstein warped product with $B$ compact and connected, ${\rm dim}B=n_{1}\geqslant2,\;{\rm dim}F=n_{2}\geqslant2,\;{\rm dim}M=\overline n=n_{1}+n_{2},\;P\in\Gamma(TB)$ and $U\in\Gamma(TB).$
If $a\geqslant0$ and $Pf\geqslant0,$ then $M$ is a sample Riemannian product.
\end{Theorem}
\begin{proof}
Let $z\in B$ such that $f(z)$ is the maximum of $f$ on $B.$ Then $grad_{B}f(z)=0$ and $\Delta_{B}f(z)\geqslant0.$
So $Pf(z)=g(grad_{B}f(z),P)=0.$
Writing the equation $(15)$ in the point $z$ we obtain:
$$f(z)\Delta_{B}f(z)+af^{2}(z)=\alpha_{1}.   \eqno{(17)}$$
By equations $(15)$ and $(17),$ we have:
$$ f(z)\Delta_{B}f(z)+af^{2}(z)=f\Delta_{B}f+(1-n_{2})|grad_{B}f|_{B}^{2}+(1-\overline{n})fPf+af^{2}$$
$$ \Rightarrow f\Delta_{B}f=f(z)\Delta_{B}f(z)+(n_{2}-1)|grad_{B}f|_{B}^{2}+(\overline{n}-1)fPf+a\big[f^{2}(z)-f^{2}\big]\geqslant0$$
Then the Laplacian has constant sign and hence $f$ is constant.
\end{proof}

\begin{Theorem}
Let $M=B \times_{f}F$ be a quasi-Einstein warped product with $B$ compact and connected, ${\rm dim}B=n_{1}\geqslant2,\;{\rm dim}F=n_{2}\geqslant2,\;{\rm dim}M=\overline n=n_{1}+n_{2},\;P\in\Gamma(TB)$ and $U\in\Gamma(TB).$
If $a<0,\;Pf\geqslant0,$ and $\overline{S}^{F}\geqslant0$, then $M$ is a sample Riemannian product.
\end{Theorem}
\begin{proof}
From the third equation of $(11)$ and the equation $(15),$ we get $\overline{S}^{F}=n_{2}\alpha_{1}.$ Since $\overline{S}^{F}\geqslant0,$ it follows
that $\alpha_{1}\geqslant0.$ Then the equation $(15)$ becomes:
$$f\Delta_{B}f+af^{2}=\alpha_{1}+(n_{2}-1)|grad_{B}f|_{B}^{2}+(\overline{n}-1)fPf\geqslant0$$
$$\Rightarrow f\Delta_{B}f+af^{2}\geqslant0 \Rightarrow\Delta_{B}f\geqslant-af>0.$$
So $f$ is constant.
\end{proof}
\subsubsection{\normalsize{\bfseries When \boldmath $P\in\Gamma(TB),\;U\in\Gamma(TF).$}}
\begin{Theorem}
Let $M=B \times_{f}F$ be a quasi-Einstein warped product with $B$ compact and connected, ${\rm dim}B=n_{1}\geqslant1,\;{\rm dim}F=n_{2}\geqslant2,\;{\rm dim}M=\overline n=n_{1}+n_{2},\;P\in\Gamma(TB)$ and $U\in\Gamma(TF).$
If $n_{1}=1,\;div_{B}P=c_{1}$ and $\pi(P)=c_{2}$ are both constants, then $M$ is a sample Riemannian product.
\end{Theorem}
\begin{proof}
If $n_{1}=1,$ then $\overline{S}^{B}=0,$ by the second equation of $(12)$ and $div_{B}P=c_{1},\;\pi(P)=c_{2}$
we can get
$$0=a+n_{2}\frac{\Delta_{B}f}{f}-n_{2}c_{1}+n_{2}c_{2}\Rightarrow \frac{\Delta_{B}f}{f}=c_{1}-c_{2}-\frac{a}{n_{2}}\triangleq c_{0}\Rightarrow \Delta_{B}f=c_{0}f.$$
Then the Laplacian has constant sign and hence $f$ is constant.
\end{proof}
\begin{Theorem}
Let $M=B \times_{f}F$ be a quasi-Einstein warped product, ${\rm dim}B=n_{1}\geqslant2,\;{\rm dim}F=n_{2}\geqslant2,\;{\rm dim}M=\overline n=n_{1}+n_{2},\;P\in\Gamma(TB)$ and $U\in\Gamma(TF).$
If $b\neq0,$ then $M$ is a sample Riemannian product.
\end{Theorem}
\begin{proof}
Consider in the second equation of $(8)$ that $V,W$ are orthogonal vector fields tangent to $F$ such that $g_{M}(V,U)\neq0,$ and $g_{M}(W,U)\neq0.$  Then $g_{F}(V,W)=0,$ and we have
$$\overline{Ric}^{F}(V,W)=bf^{4}g_{F}(V,U)g_{F}(W,U).  \eqno{(18)}$$
Taking in consideration the different domains of definition of the functions that appear in the equation (18), we obtain that $f$ is constant.
\end{proof}
\subsubsection{\normalsize {\bfseries When \boldmath $P\in\Gamma(TF),\;U\in\Gamma(TB).$}}
In the second equation of $(9),$ let:
$$f\Delta_{B}f+(1-n_{2})|grad_{B}f|_{B}^{2}+af^{2}\triangleq \alpha_{2}.   \eqno{(19)} $$
Then the equation $(9)$ becomes:
\begin{equation}
\begin{cases}\setcounter{equation}{20}
\overline{{Ric}}^{B}(X,Y)=ag_{B}(X,Y)-n_{2}\frac{H^f_{B}(X,Y)}{f}+bg_{B}(X,U)g_{B}(Y,U),   \\
\overline{Ric}^{F}(V,W)=\alpha_{2}g_{F}(V,W)+(1-\overline n)g(W,\nabla_{V}P)+(\overline n-1)\pi(V)\pi(W),  \\
f\Delta_{B}f+(1-n_{2})|grad_{B}f|_{B}^{2}+af^{2}=\alpha_{2}.
\end{cases}
\end{equation}
\begin{Theorem}
Let $M=B \times_{f}F$ be a quasi-Einstein warped product with $B$ compact and connected, ${\rm dim}B=n_{1}\geqslant1,\;{\rm dim}F=n_{2}\geqslant2,\;{\rm dim}M=\overline n=n_{1}+n_{2},\;P\in\Gamma(TF)$ and $U\in\Gamma(TB).$
If $n_{1}=1,$ then $M$ is a sample Riemannian product.
\end{Theorem}
\begin{proof}
If $n_{1}=1,$ then ${S}^{B}=0,$ by the second equation of $(13)$ we can get
$$0=a+n_{2}\frac{\Delta_{B}f}{f}+b\Rightarrow \frac{\Delta_{B}f}{f}=-\frac{a+b}{n_{2}}\triangleq c_{0}\Rightarrow
\Delta_{B}f=c_{0}f.$$
Then the Laplacian has constant sign and hence $f$ is constant.
\end{proof}

\begin{Theorem}
Let $M=B \times_{f}F$ be a quasi-Einstein warped product with $B$ compact and connected, ${\rm dim}B=n_{1}\geqslant2,\;{\rm dim}F=n_{2}\geqslant2,\;{\rm dim}M=\overline n=n_{1}+n_{2},\;P\in\Gamma(TF)$ and $U\in\Gamma(TB).$
If $a\geqslant0,$ then $M$ is a sample Riemannian product.
\end{Theorem}
\begin{proof}
Let $z\in B$ such that $f(z)$ is the maximum of $f$ on $B.$ Then $grad_{B}f(z)=0$ and $\Delta_{B}f(z)\geqslant0.$
Writing the equation $(19)$ in the point $z$ we obtain:
$$f(z)\Delta_{B}f(z)+af^{2}(z)=\alpha_{2}.   \eqno{(21)}$$
By equations $(19)$ and $(21),$ we have:
$$ f(z)\Delta_{B}f(z)+af^{2}(z)=f\Delta_{B}f+(1-n_{2})|grad_{B}f|_{B}^{2}+af^{2}$$
$$ \Rightarrow f\Delta_{B}f=f(z)\Delta_{B}f(z)+(n_{2}-1)|grad_{B}f|_{B}^{2}+a\big[f^{2}(z)-f^{2}\big]\geqslant0$$
Then the Laplacian has constant sign and hence $f$ is constant.
\end{proof}

\begin{Theorem}
Let $M=B \times_{f}F$ be a quasi-Einstein warped product with $B$ compact and connected, ${\rm dim}B=n_{1}\geqslant2,\;{\rm dim}F=n_{2}\geqslant2,\;{\rm dim}M=\overline n=n_{1}+n_{2},\;P\in\Gamma(TF)$ and $U\in\Gamma(TB).$
If $a<0,\;div_{F}P=0,\;\pi(P)=0,$ and $S^{F}\geqslant0$, then $M$ is a sample Riemannian product.
\end{Theorem}
\begin{proof}
Consider that $div_{F}P=0,\;\pi(P)=0,$ then from the third equation of $(13)$ and the equation $(19),$ we get $S^{F}=n_{2}\alpha_{2}.$ Since $S^{F}\geqslant0,$ it follows
that $\alpha_{2}\geqslant0.$ Then the equation $(19)$ becomes:
$$f\Delta_{B}f+af^{2}=\alpha_{2}+(n_{2}-1)|grad_{B}f|_{B}^{2}\geqslant0
\Rightarrow f\Delta_{B}f+af^{2}\geqslant0 \Rightarrow\Delta_{B}f\geqslant-af>0.$$
So $f$ is constant.
\end{proof}
\subsubsection{\normalsize {\bfseries When \boldmath $P\in\Gamma(TF),\;U\in\Gamma(TF).$}}
\begin{Theorem}
Let $M=B \times_{f}F$ be a quasi-Einstein warped product with $B$ compact and connected, ${\rm dim}B=n_{1}\geqslant1,\;{\rm dim}F=n_{2}\geqslant2,\;{\rm dim}M=\overline n=n_{1}+n_{2},\;P\in\Gamma(TF)$ and $U\in\Gamma(TF).$
If $n_{1}=1,$ then $M$ is a sample Riemannian product.
\end{Theorem}
\begin{proof}
If $n_{1}=1,$ then ${S}^{B}=0,$ by the second equation of $(14)$ we can get
$$0=a+n_{2}\frac{\Delta_{B}f}{f}\Rightarrow \frac{\Delta_{B}f}{f}=-\frac{a}{n_{2}}\triangleq c_{0}\Rightarrow
\Delta_{B}f=c_{0}f.$$
Then the Laplacian has constant sign and hence $f$ is constant.
\end{proof}
\begin{Theorem}
Let $M=B \times_{f}F$ be a quasi-Einstein warped product, ${\rm dim}B=n_{1}\geqslant2,\;{\rm dim}F=n_{2}\geqslant2,\;{\rm dim}M=\overline n=n_{1}+n_{2},\;P\in\Gamma(TF)$ and $U\in\Gamma(TF).$
If $b\neq0,$ and $P$ is parallel on $F$ with respect to the Levi-Civita connection on $F,$ then $M$ is a sample Riemannian product.
\end{Theorem}
\begin{proof}
Consider in the second equation of $(10)$ that $V,W$ are orthogonal vector fields tangent to $F$ such that $g_{M}(V,U)\neq0,\;g_{M}(W,U)\neq0,$ and $\pi(W)=0.$ Then $g_{F}(V,W)=0.$ Since $P$ is parallel on $F$ with respect to the Levi-Civita connection on $F,$ we have $\nabla_{V}P=0,$ so the second equation of $(10)$ becomes:
$$Ric^{F}(V,W)=bf^{4}g_{F}(V,U)g_{F}(W,U).  \eqno{(22)}$$
Taking in consideration the different domains of definition of the functions that appear in the equation (22), we obtain that $f$ is constant.
\end{proof}

\section{\large Generalized quasi-Einstein warped products with a semi-symmetric non-metric connection}
\subsection{\large Ricci curvature and scalar curvature}
In this section, we compute Ricci curvature and scalar curvature of generalized quasi-Einstein warped products with a semi-symmetric non-metric connection.

\begin{Proposition}
Let $M=B \times_{f}F$ be a generalized quasi-Einstein warped product, ${\rm dim}B=n_{1},\;{\rm dim}F=n_{2},\;{\rm dim}M=\overline n=n_{1}+n_{2},\;X,Y\in\Gamma(TB),\;V,W\in\Gamma(TF),$ and $P\in\Gamma(TB).$\\
$1)$ If $U_{1}\in\Gamma(TB),\;U_{2}\in\Gamma(TB),$ then
\begin{equation}
\begin{cases}\setcounter{equation}{23}
\overline{Ric}^{B}(X,Y)=ag_{B}(X,Y)-n_{2}\Big[\frac{H^f_{B}(X,Y)}{f}+g(Y,\nabla_{X}P)-\pi(X)\pi(Y)\Big]
\\ \mbox{}\quad\quad\quad\quad\quad\quad+bg_{B}(X,U_{1})g_{B}(Y,U_{1})+cg_{B}(X,U_{2})g_{B}(Y,U_{2}),\\
\overline{Ric}^{F}(V,W)=g_{F}(V,W)\big[f\Delta_{B}f+(1-n_{2})|grad_{B}f|_{B}^{2}+(1-\overline{n})fPf+af^{2}\big].
\end{cases}
\end{equation}
$2)$ If $U_{1}\in\Gamma(TF),\;U_{2}\in\Gamma(TF),$ then
\begin{equation}
\begin{cases}\setcounter{equation}{24}
\overline{Ric}^{B}(X,Y)=ag_{B}(X,Y)-n_{2}\Big[\frac{H^f_{B}(X,Y)}{f}+g(Y,\nabla_{X}P)-\pi(X)\pi(Y)\Big],\\
\overline{Ric}^{F}(V,W)=g_{F}(V,W)\big[f\Delta_{B}f+(1-n_{2})|grad_{B}f|_{B}^{2}+(1-\overline{n})fPf+af^{2}\big]\\
\mbox{}\quad\quad\quad\quad\quad\quad+bf^{4}g_{F}(V,U_{1})g_{F}(W,U_{1})+cf^{4}g_{F}(V,U_{2})g_{F}(W,U_{2}).
\end{cases}
\end{equation}
$3)$ If $U_{1}\in\Gamma(TB),\;U_{2}\in\Gamma(TF),$ then
\begin{equation}
\begin{cases}\setcounter{equation}{25}
\overline{Ric}^{B}(X,Y)=ag_{B}(X,Y)-n_{2}\Big[\frac{H^f_{B}(X,Y)}{f}+g(Y,\nabla_{X}P)-\pi(X)\pi(Y)\Big]
\\ \mbox{}\quad\quad\quad\quad\quad\quad+bg_{B}(X,U_{1})g_{B}(Y,U_{1}),\\
\overline{Ric}^{F}(V,W)=g_{F}(V,W)\big[f\Delta_{B}f+(1-n_{2})|grad_{B}f|_{B}^{2}+(1-\overline{n})fPf+af^{2}\big]\\
\mbox{}\quad\quad\quad\quad\quad\quad+cf^{4}g_{F}(V,U_{2})g_{F}(W,U_{2}).
\end{cases}
\end{equation}
\end{Proposition}
\begin{proof}
$1)$ Since $M$ is generalized quasi-Einstein, we have $\overline{Ric}(X,Y)=ag_{B}(X,Y)+bg_{B}(X,U_{1})g_{B}(Y,U_{1})\\+cg_{B}(X,U_{2})g_{B}(Y,U_{2}),$ then use Lemma 2.6(1) we get the first equation of $(23).$\\
Since $M$ is generalized quasi-Einstein, we have
$\overline{Ric}(V,W)=af^{2}g_{F}(V,W)+bg(V,U_{1})g(W,U_{1})+cg(V,U_{2})g(W,U_{2})=af^{2}g_{F}(V,W),$ then use Lemma 2.6(3) we get the second equation of $(23).$\\
$2)$ Since $M$ is generalize quasi-Einstein, we have
$\overline{Ric}(X,Y)=ag_{B}(X,Y)+bg(X,U_{1})g(Y,U_{1})+cg(X,U_{2})g(Y,U_{2})=ag_{B}(X,Y),$ then use Lemma 2.6(1) we get the first equation of $(24).$\\
Since $M$ is generalize quasi-Einstein, we have
$\overline{Ric}(V,W)=af^{2}g_{F}(V,W)+bf^{4}g_{F}(V,U_{1})g_{F}(W,U_{1})+cf^{4}g_{F}(V,U_{2})g_{F}(W,U_{2}),$ then use Lemma 2.6(3) we get the second equation of $(24).$\\
$3)$ Since $M$ is generalize quasi-Einstein, we have $\overline{Ric}(X,Y)=ag_{B}(X,Y)+bg(X,U_{1})g(Y,U_{1})+cg(X,U_{2})g(Y,U_{2})=ag_{B}(X,Y)+bg_{B}(X,U_{1})g_{B}(Y,U_{1}),$ then use Lemma 2.6(1) we get the first equation of $(25).$\\
Since $M$ is generalize quasi-Einstein, we have
$\overline{Ric}(V,W)=af^{2}g_{F}(V,W)+bg(V,U_{1})g(W,U_{1})+cg(V,U_{2})g(W,U_{2})=
af^{2}g_{F}(V,W)+cf^{4}g_{F}(V,U_{2})g_{F}(W,U_{2}),$ then use Lemma 2.6(3) we get the second equation of $(25).$
\end{proof}
Using the same method, we can get the following Proposition:
\begin{Proposition}
Let $M=B \times_{f}F$ be a generalized quasi-Einstein warped product, ${\rm dim}B=n_{1},\;{\rm dim}F=n_{2},\;{\rm dim}M=\overline n=n_{1}+n_{2},\;X,Y\in\Gamma(TB),\;V,W\in\Gamma(TF),$ and $P\in\Gamma(TF).$\\
$1)$ If $U_{1}\in\Gamma(TB),\;U_{2}\in\Gamma(TB),$ then
\begin{equation}
\begin{cases}\setcounter{equation}{26}
Ric^{B}(X,Y)=ag_{B}(X,Y)-n_{2}\frac{H^f_{B}(X,Y)}{f}+bg_{B}(X,U_{1})g_{B}(Y,U_{1})+cg_{B}(X,U_{2})g_{B}(Y,U_{2}),  \\
Ric^{F}(V,W)=g_{F}(V,W)\big[f\Delta_{B}f+(1-n_{2})|grad_{B}f|_{B}^{2}+af^{2}\big]
+(1-\overline n)g(W,\nabla_{V}P)\\ \mbox{}\quad\quad\quad\quad\quad\quad+(\overline n-1)\pi(V)\pi(W).
\end{cases}
\end{equation}
$2)$ If $U_{1}\in\Gamma(TF),\;U_{2}\in\Gamma(TF),$ then
\begin{equation}
\begin{cases}\setcounter{equation}{27}
Ric^{B}(X,Y)=ag_{B}(X,Y)-n_{2}\frac{H^f_{B}(X,Y)}{f},   \\
Ric^{F}(V,W)=g_{F}(V,W)\big[f\Delta_{B}f+(1-n_{2})|grad_{B}f|_{B}^{2}+af^{2}\big]+(1-\overline n)g(W,\nabla_{V}P)\\ \mbox{}\quad\quad\quad\quad\quad+(\overline n-1)\pi(V)\pi(W)
+bf^{4}g_{F}(V,U_{1})g_{F}(W,U_{1})+cf^{4}g_{F}(V,U_{2})g_{F}(W,U_{2}).
\end{cases}
\end{equation}
$3)$ If $U_{1}\in\Gamma(TB),\;U_{2}\in\Gamma(TF),$ then
\begin{equation}
\begin{cases}\setcounter{equation}{28}
Ric^{B}(X,Y)=ag_{B}(X,Y)-n_{2}\frac{H^f_{B}(X,Y)}{f}+bg_{B}(X,U_{1})g_{B}(Y,U_{1}),   \\
Ric^{F}(V,W)=g_{F}(V,W)\big[f\Delta_{B}f+(1-n_{2})|grad_{B}f|_{B}^{2}+af^{2}\big]+(1-\overline n)g(W,\nabla_{V}P)\\ \mbox{}\quad\quad\quad\quad\quad+(\overline n-1)\pi(V)\pi(W)+cf^{4}g_{F}(V,U_{2})g_{F}(W,U_{2}).
\end{cases}
\end{equation}
\end{Proposition}
By Lemma 2.8, Proposition 4.1 and the definition of scalar curvature, we have:
\begin{Proposition}
Let $M=B \times_{f}F$ be a generalized quasi-Einstein warped product, ${\rm dim}B=n_{1},\;{\rm dim}F=n_{2},\;{\rm dim}M=\overline n=n_{1}+n_{2},\;X,Y\in\Gamma(TB),\;V,W\in\Gamma(TF),$ and $P\in\Gamma(TB).$\\
$1)$ If $U_{1}\in\Gamma(TB),\;U_{2}\in\Gamma(TB),$ then
\begin{equation}
\begin{cases}
\overline{S}^{M}=\overline{n}a+b+c;    \\
\overline{S}^{B}=n_{1}a+n_{2}\frac{\Delta_{B}f}{f}-n_{2}div_{B}P+n_{2}\pi(P)+b+c;\\
\overline{S}^{F}=n_{2}\big[f\Delta_{B}f+(1-n_{2})|grad_{B}f|_{B}^{2}+(1-\overline{n})fPf+af^{2}\big].
\end{cases}
\end{equation}
$2)$ If $U_{1}\in\Gamma(TF),\;U_{2}\in\Gamma(TF),$ then
\begin{equation}
\begin{cases}
\overline{S}^{M}=\overline{n}a+b+c;    \\
\overline{S}^{B}=n_{1}a+n_{2}\frac{\Delta_{B}f}{f}-n_{2}div_{B}P+n_{2}\pi(P);\\
\overline{S}^{F}=n_{2}\big[f\Delta_{B}f+(1-n_{2})|grad_{B}f|_{B}^{2}+(1-\overline{n})fPf+af^{2}\big]+bf^{2}+cf^{2}.
\end{cases}
\end{equation}
$3)$ If $U_{1}\in\Gamma(TB),\;U_{2}\in\Gamma(TF),$ then
\begin{equation}
\begin{cases}
\overline{S}^{M}=\overline{n}a+b+c;    \\
\overline{S}^{B}=n_{1}a+n_{2}\frac{\Delta_{B}f}{f}-n_{2}div_{B}P+n_{2}\pi(P)+b;\\
\overline{S}^{F}=n_{2}\big[f\Delta_{B}f+(1-n_{2})|grad_{B}f|_{B}^{2}+(1-\overline{n})fPf+af^{2}\big]+cf^{2}.
\end{cases}
\end{equation}
\end{Proposition}
Similarly we have:
\begin{Proposition}
Let $M=B \times_{f}F$ be a generalized quasi-Einstein warped product, ${\rm dim}B=n_{1},\;{\rm dim}F=n_{2},\;{\rm dim}M=\overline n=n_{1}+n_{2},\;X,Y\in\Gamma(TB),\;V,W\in\Gamma(TF),$ and $P\in\Gamma(TF).$\\
$1)$ If $U_{1}\in\Gamma(TB),\;U_{2}\in\Gamma(TB),$ then
\begin{equation}
\begin{cases}
\overline{S}^{M}=\overline{n}a+b+c;    \\
S^{B}=n_{1}a+n_{2}\frac{\Delta_{B}f}{f}+b+c;\\
S^{F}=n_{2}\big[f\Delta_{B}f+(1-n_{2})|grad_{B}f|_{B}^{2}+af^{2}\big]+(1-\overline{n})div_{F}P
+(\overline{n}-1)\pi(P).
\end{cases}
\end{equation}
$2)$ If $U_{1}\in\Gamma(TF),\;U_{2}\in\Gamma(TF),$ then
\begin{equation}
\begin{cases}
\overline{S}^{M}=\overline{n}a+b+c;    \\
S^{B}=n_{1}a+n_{2}\frac{\Delta_{B}f}{f};\\
S^{F}=n_{2}\big[f\Delta_{B}f+(1-n_{2})|grad_{B}f|_{B}^{2}+af^{2}\big]+(1-\overline{n})div_{F}P
+(\overline{n}-1)\pi(P)+bf^{2}+cf^{2}.
\end{cases}
\end{equation}
$3)$ If $U_{1}\in\Gamma(TB),\;U_{2}\in\Gamma(TF),$ then
\begin{equation}
\begin{cases}
\overline{S}^{M}=\overline{n}a+b+c;    \\
S^{B}=n_{1}a+n_{2}\frac{\Delta_{B}f}{f}+b;\\
S^{F}=n_{2}\big[f\Delta_{B}f+(1-n_{2})|grad_{B}f|_{B}^{2}+af^{2}\big]+(1-\overline{n})div_{F}P
+(\overline{n}-1)\pi(P)+cf^{2}.
\end{cases}
\end{equation}
\end{Proposition}

\subsection{\large Obstructions to the existence of generalized quasi-Einstein warped products with a semi-symmetric non-metric connection}
In this section, we prove some obstructions to the existence of generalized quasi-Einstein warped products with a semi-symmetric non-metric connection. We consider the following six cases:
\subsubsection{\normalsize{\bfseries When \boldmath $P\in\Gamma(TB),\;U_{1}\in\Gamma(TB),\;U_{2}\in\Gamma(TB).$}}
In the second equation of $(23),$ let:
$$f\Delta_{B}f+(1-n_{2})|grad_{B}f|_{B}^{2}+(1-\overline{n})fPf+af^{2}\triangleq \alpha_{3}.   \eqno{(35)} $$
Then the equation $(23)$ becomes:
\begin{equation}
\begin{cases}\setcounter{equation}{36}
\overline{{Ric}}^{B}(X,Y)=ag_{B}(X,Y)-n_{2}\Big[\frac{H^f_{B}(X,Y)}{f}+g(Y,\nabla_{X}P)-\pi(X)\pi(Y)\Big]
\\ \mbox{}\quad\quad\quad\quad\quad\quad+bg_{B}(X,U_{1})g_{B}(Y,U_{1})+cg_{B}(X,U_{2})g_{B}(Y,U_{2}),   \\
\overline{Ric}^{F}(V,W)=\alpha_{3} g_{F}(V,W),  \\
f\Delta_{B}f+(1-n_{2})|grad_{B}f|_{B}^{2}+(1-\overline{n})fPf+af^{2}=\alpha_{3}.
\end{cases}
\end{equation}
\begin{Theorem}
Let $M=B \times_{f}F$ be a generalized quasi-Einstein warped product with $B$ compact and connected, ${\rm dim}B=n_{1}\geqslant1,\;{\rm dim}F=n_{2}\geqslant2,\;{\rm dim}M=\overline n=n_{1}+n_{2},\;P\in\Gamma(TB)$ and $U_{1}\in\Gamma(TB),\;U_{2}\in\Gamma(TB).$
If $n_{1}=1,\;div_{B}P=c_{1}$ and $\pi(P)=c_{2}$ are both constants, then $M$ is a sample Riemannian product.
\end{Theorem}
\begin{proof}
If $n_{1}=1,$ then $\overline{S}^{B}=0,$ by the second equation of $(29)$ and $div_{B}P=c_{1},\;\pi(P)=c_{2}$
we get
$$0=a+n_{2}\frac{\Delta_{B}f}{f}-n_{2}c_{1}+n_{2}c_{2}+b+c\Rightarrow \frac{\Delta_{B}f}{f}=c_{1}-c_{2}-\frac{a+b+c}{n_{2}}\triangleq c_{0}\Rightarrow \Delta_{B}f=c_{0}f.$$
Then the Laplacian has constant sign and hence $f$ is constant.
\end{proof}

\begin{Theorem}
Let $M=B \times_{f}F$ be a generalized quasi-Einstein warped product with $B$ compact and connected, ${\rm dim}B=n_{1}\geqslant2,\;{\rm dim}F=n_{2}\geqslant2,\;{\rm dim}M=\overline n=n_{1}+n_{2},\;P\in\Gamma(TB)$ and $U_{1}\in\Gamma(TB),\;U_{2}\in\Gamma(TB).$
If $a\geqslant0$ and $Pf\geqslant0,$ then $M$ is a sample Riemannian product.
\end{Theorem}
\begin{proof}
Let $z\in B$ such that $f(z)$ is the maximum of $f$ on $B.$ Then $grad_{B}f(z)=0$ and $\Delta_{B}f(z)\geqslant0.$
So $Pf(z)=g(grad_{B}f(z),P)=0.$
Writing the equation $(35)$ in the point $z$ we obtain:
$$f(z)\Delta_{B}f(z)+af^{2}(z)=\alpha_{3}.   \eqno{(37)}$$
By equations $(35)$ and $(37),$ we have:
$$ f(z)\Delta_{B}f(z)+af^{2}(z)=f\Delta_{B}f+(1-n_{2})|grad_{B}f|_{B}^{2}+(1-\overline{n})fPf+af^{2}$$
$$ \Rightarrow f\Delta_{B}f=f(z)\Delta_{B}f(z)+(n_{2}-1)|grad_{B}f|_{B}^{2}+(\overline{n}-1)fPf+a\big[f^{2}(z)-f^{2}\big]\geqslant0$$
Then the Laplacian has constant sign and hence $f$ is constant.
\end{proof}

\begin{Theorem}
Let $M=B \times_{f}F$ be a generalized quasi-Einstein warped product with $B$ compact and connected, ${\rm dim}B=n_{1}\geqslant2,\;{\rm dim}F=n_{2}\geqslant2,\;{\rm dim}M=\overline n=n_{1}+n_{2},\;P\in\Gamma(TB)$ and $U_{1}\in\Gamma(TB),\;U_{2}\in\Gamma(TB).$
If $a<0,\;Pf\geqslant0,$ and $\overline{S}^{F}\geqslant0$, then $M$ is a sample Riemannian product.
\end{Theorem}
\begin{proof}
From the third equation of $(29)$ and the equation $(35),$ we get $\overline{S}^{F}=n_{2}\alpha_{3}.$ Since $\overline{S}^{F}\geqslant0,$ it follows
that $\alpha_{3}\geqslant0.$ Then the equation $(35)$ becomes:
$$f\Delta_{B}f+af^{2}=\alpha_{3}+(n_{2}-1)|grad_{B}f|_{B}^{2}+(\overline{n}-1)fPf\geqslant0$$
$$\Rightarrow f\Delta_{B}f+af^{2}\geqslant0 \Rightarrow\Delta_{B}f\geqslant-af>0.$$
So $f$ is constant.
\end{proof}
\subsubsection{\normalsize{\bfseries When \boldmath $P\in\Gamma(TB),\;U_{1}\in\Gamma(TF),\;U_{2}\in\Gamma(TF).$}}
\begin{Theorem}
Let $M=B \times_{f}F$ be a generalized quasi-Einstein warped product with $B$ compact and connected, ${\rm dim}B=n_{1}\geqslant1,\;{\rm dim}F=n_{2}\geqslant2,\;{\rm dim}M=\overline n=n_{1}+n_{2},\;P\in\Gamma(TB)$ and $U_{1}\in\Gamma(TF),\;U_{2}\in\Gamma(TF).$
If $n_{1}=1,\;div_{B}P=c_{1}$ and $\pi(P)=c_{2}$ are both constants, then $M$ is a sample Riemannian product.
\end{Theorem}
\begin{proof}
If $n_{1}=1,$ then $\overline{S}^{B}=0,$ by the second equation of $(30)$ and $div_{B}P=c_{1},\;\pi(P)=c_{2}$
we can get
$$0=a+n_{2}\frac{\Delta_{B}f}{f}-n_{2}c_{1}+n_{2}c_{2}\Rightarrow \frac{\Delta_{B}f}{f}=c_{1}-c_{2}-\frac{a}{n_{2}}\triangleq c_{0}\Rightarrow \Delta_{B}f=c_{0}f.$$
Then the Laplacian has constant sign and hence $f$ is constant.
\end{proof}
\begin{Theorem}
Let $M=B \times_{f}F$ be a generalized quasi-Einstein warped product, ${\rm dim}B=n_{1}\geqslant2,\;{\rm dim}F=n_{2}\geqslant2,\;{\rm dim}M=\overline n=n_{1}+n_{2},\;P\in\Gamma(TB)$ and $U_{1}\in\Gamma(TF),\;U_{2}\in\Gamma(TF).$ If $b\neq0$ or $c\neq0,$ then $M$ is a sample Riemannian product.
\end{Theorem}
\begin{proof}
Consider in the second equation of $(24)$ that $V,W$ are orthogonal vector fields tangent to $F$ such that $g_{M}(V,U_{1})\neq0,\;g_{M}(W,U_{1})\neq0$ or $g_{M}(V,U_{2})\neq0,\;g_{M}(W,U_{2})\neq0.$  Then $g_{F}(V,W)=0,$ and we have
$$\overline{Ric}^{F}(V,W)=bf^{4}g_{F}(V,U_{1})g_{F}(W,U_{1}),  \eqno{(38)}$$ or
$$\overline{Ric}^{F}(V,W)=cf^{4}g_{F}(V,U_{2})g_{F}(W,U_{2}).  \eqno{(39)}$$
Taking in consideration the different domains of definition of the functions that appear in the equation (38) or (39), we obtain that $f$ is constant.
\end{proof}

\subsubsection{\normalsize{\bfseries When \boldmath $P\in\Gamma(TB),\;U_{1}\in\Gamma(TB),\;U_{2}\in\Gamma(TF).$}}
\begin{Theorem}
Let $M=B \times_{f}F$ be a generalized quasi-Einstein warped product with $B$ compact and connected, ${\rm dim}B=n_{1}\geqslant1,\;{\rm dim}F=n_{2}\geqslant2,\;{\rm dim}M=\overline n=n_{1}+n_{2},\;P\in\Gamma(TB)$ and $U_{1}\in\Gamma(TB),\;U_{2}\in\Gamma(TF).$
If $n_{1}=1,\;div_{B}P=c_{1}$ and $\pi(P)=c_{2}$ are both constants, then $M$ is a sample Riemannian product.
\end{Theorem}
\begin{proof}
If $n_{1}=1,$ then $\overline{S}^{B}=0,$ by the second equation of $(31)$ and $div_{B}P=c_{1},\;\pi(P)=c_{2}$ we get
$$0=a+n_{2}\frac{\Delta_{B}f}{f}-n_{2}c_{1}+n_{2}c_{2}+b\Rightarrow \frac{\Delta_{B}f}{f}=c_{1}-c_{2}-\frac{a+b}{n_{2}}\triangleq c_{0}\Rightarrow \Delta_{B}f=c_{0}f.$$
Then the Laplacian has constant sign and hence $f$ is constant.
\end{proof}
\begin{Theorem}
Let $M=B \times_{f}F$ be a generalized quasi-Einstein warped product, ${\rm dim}B=n_{1}\geqslant2,\;{\rm dim}F=n_{2}\geqslant2,\;{\rm dim}M=\overline n=n_{1}+n_{2},\;P\in\Gamma(TB)$ and $U_{1}\in\Gamma(TB),\;U_{2}\in\Gamma(TF).$ If $c\neq0,$ then $M$ is a sample Riemannian product.
\end{Theorem}
\begin{proof}
Consider in the second equation of $(25)$ that $V,W$ are orthogonal vector fields tangent to $F$ such that $g_{M}(V,U_{2})\neq0,\;g_{M}(W,U_{2})\neq0.$  Then $g_{F}(V,W)=0,$ and we have
$$\overline{Ric}^{F}(V,W)=cf^{4}g_{F}(V,U_{2})g_{F}(W,U_{2}).  \eqno{(40)}$$
Taking in consideration the different domains of definition of the functions that appear in the equation (40), we obtain that $f$ is constant.
\end{proof}

\subsubsection{\normalsize {\bfseries When \boldmath $P\in\Gamma(TF),\;U_{1}\in\Gamma(TB),\;U_{2}\in\Gamma(TB).$}}
In the second equation of $(26),$ let:
$$f\Delta_{B}f+(1-n_{2})|grad_{B}f|_{B}^{2}+af^{2}\triangleq \alpha_{4}.   \eqno{(41)} $$
Then the equation $(26)$ becomes:
\begin{equation}
\begin{cases}\setcounter{equation}{42}
\overline{{Ric}}^{B}(X,Y)=ag_{B}(X,Y)-n_{2}\frac{H^f_{B}(X,Y)}{f}
+bg_{B}(X,U_{1})g_{B}(Y,U_{1})+cg_{B}(X,U_{2})g_{B}(Y,U_{2}),   \\
\overline{Ric}^{F}(V,W)=\alpha_{4}g_{F}(V,W)+(1-\overline n)g(W,\nabla_{V}P)+(\overline n-1)\pi(V)\pi(W),  \\
f\Delta_{B}f+(1-n_{2})|grad_{B}f|_{B}^{2}+af^{2}=\alpha_{4}.
\end{cases}
\end{equation}
\begin{Theorem}
Let $M=B \times_{f}F$ be a generalized quasi-Einstein warped product with $B$ compact and connected, ${\rm dim}B=n_{1}\geqslant1,\;{\rm dim}F=n_{2}\geqslant2,\;{\rm dim}M=\overline n=n_{1}+n_{2},\;P\in\Gamma(TF)$ and $U_{1}\in\Gamma(TB),\;U_{2}\in\Gamma(TB).$
If $n_{1}=1,$ then $M$ is a sample Riemannian product.
\end{Theorem}
\begin{proof}
If $n_{1}=1,$ then ${S}^{B}=0,$ by the second equation of $(32)$ we can get
$$0=a+n_{2}\frac{\Delta_{B}f}{f}+b+c\Rightarrow \frac{\Delta_{B}f}{f}=-\frac{a+b+c}{n_{2}}\triangleq c_{0}\Rightarrow
\Delta_{B}f=c_{0}f.$$
Then the Laplacian has constant sign and hence $f$ is constant.
\end{proof}

\begin{Theorem}
Let $M=B \times_{f}F$ be a generalized quasi-Einstein warped product with $B$ compact and connected, ${\rm dim}B=n_{1}\geqslant2,\;{\rm dim}F=n_{2}\geqslant2,\;{\rm dim}M=\overline n=n_{1}+n_{2},\;P\in\Gamma(TF)$ and $U_{1}\in\Gamma(TB),\;U_{2}\in\Gamma(TB).$
If $a\geqslant0,$ then $M$ is a sample Riemannian product.
\end{Theorem}
\begin{proof}
Let $z\in B$ such that $f(z)$ is the maximum of $f$ on $B.$ Then $grad_{B}f(z)=0$ and $\Delta_{B}f(z)\geqslant0.$
Writing the equation $(41)$ in the point $z$ we obtain:
$$f(z)\Delta_{B}f(z)+af^{2}(z)=\alpha_{4}.   \eqno{(43)}$$
By equations $(41)$ and $(43),$ we have:
$$ f(z)\Delta_{B}f(z)+af^{2}(z)=f\Delta_{B}f+(1-n_{2})|grad_{B}f|_{B}^{2}+af^{2}$$
$$ \Rightarrow f\Delta_{B}f=f(z)\Delta_{B}f(z)+(n_{2}-1)|grad_{B}f|_{B}^{2}+a\big[f^{2}(z)-f^{2}\big]\geqslant0$$
Then the Laplacian has constant sign and hence $f$ is constant.
\end{proof}

\begin{Theorem}
Let $M=B \times_{f}F$ be a generalized quasi-Einstein warped product with $B$ compact and connected, ${\rm dim}B=n_{1}\geqslant2,\;{\rm dim}F=n_{2}\geqslant2,\;{\rm dim}M=\overline n=n_{1}+n_{2},\;P\in\Gamma(TF)$ and $U_{1}\in\Gamma(TB),\;U_{2}\in\Gamma(TB).$
If $a<0,\;div_{F}P=0,\;\pi(P)=0,$ and $S^{F}\geqslant0$, then $M$ is a sample Riemannian product.
\end{Theorem}
\begin{proof}
Consider that $div_{F}P=0,\;\pi(P)=0,$ then from the third equation of $(32)$ and the equation $(41),$ we get $S^{F}=n_{2}\alpha_{4}.$ Since $S^{F}\geqslant0,$ it follows
that $\alpha_{4}\geqslant0.$ Then the equation $(41)$ becomes:
$$f\Delta_{B}f+af^{2}=\alpha_{4}+(n_{2}-1)|grad_{B}f|_{B}^{2}\geqslant0
\Rightarrow f\Delta_{B}f+af^{2}\geqslant0 \Rightarrow\Delta_{B}f\geqslant-af>0.$$
So $f$ is constant.
\end{proof}
\subsubsection{\normalsize {\bfseries When \boldmath $P\in\Gamma(TF),\;U_{1}\in\Gamma(TF),\;U_{2}\in\Gamma(TF).$}}
\begin{Theorem}
Let $M=B \times_{f}F$ be a generalized quasi-Einstein warped product with $B$ compact and connected, ${\rm dim}B=n_{1}\geqslant1,\;{\rm dim}F=n_{2}\geqslant2,\;{\rm dim}M=\overline n=n_{1}+n_{2},\;P\in\Gamma(TF)$ and $U_{1}\in\Gamma(TF),\;U_{2}\in\Gamma(TF).$
If $n_{1}=1,$ then $M$ is a sample Riemannian product.
\end{Theorem}
\begin{proof}
If $n_{1}=1,$ then ${S}^{B}=0,$ by the second equation of $(33)$ we can get
$$0=a+n_{2}\frac{\Delta_{B}f}{f}\Rightarrow \frac{\Delta_{B}f}{f}=-\frac{a}{n_{2}}\triangleq c_{0}\Rightarrow
\Delta_{B}f=c_{0}f.$$
Then the Laplacian has constant sign and hence $f$ is constant.
\end{proof}
\begin{Theorem}
Let $M=B \times_{f}F$ be a generalized quasi-Einstein warped product, ${\rm dim}B=n_{1}\geqslant2,\;{\rm dim}F=n_{2}\geqslant2,\;{\rm dim}M=\overline n=n_{1}+n_{2},\;P\in\Gamma(TF)$ and $U_{1}\in\Gamma(TF),\;U_{2}\in\Gamma(TF).$
If $b\neq0$ or $c\neq0,$ and $P$ is parallel on $F$ with respect to the Levi-Civita connection on $F,$ then $M$ is a sample Riemannian product.
\end{Theorem}
\begin{proof}
Consider in the second equation of $(27)$ that $V,W$ are orthogonal vector fields tangent to $F$ such that $g_{M}(V,U_{1})\neq0,\;g_{M}(W,U_{1})\neq0$ and $\pi(W)=0,$ or $g_{M}(V,U_{2})\neq0,\;g_{M}(W,U_{2})\neq0$ and $\pi(W)=0,$ Then $g_{F}(V,W)=0.$ Since $P$ is parallel on $F$ with respect to the Levi-Civita connection on $F,$ we have $\nabla_{V}P=0,$ so the second equation of $(27)$ becomes:
$$Ric^{F}(V,W)=bf^{4}g_{F}(V,U_{1})g_{F}(W,U_{1}),  \eqno{(44)}$$
or $$Ric^{F}(V,W)=cf^{4}g_{F}(V,U_{2})g_{F}(W,U_{2}).  \eqno{(45)}$$
Taking in consideration the different domains of definition of the functions that appear in the equation (44) or (45), we obtain that $f$ is constant.
\end{proof}

\subsubsection{\normalsize {\bfseries When \boldmath $P\in\Gamma(TF),\;U_{1}\in\Gamma(TB),\;U_{2}\in\Gamma(TF).$}}
\begin{Theorem}
Let $M=B \times_{f}F$ be a generalized quasi-Einstein warped product with $B$ compact and connected, ${\rm dim}B=n_{1}\geqslant1,\;{\rm dim}F=n_{2}\geqslant2,\;{\rm dim}M=\overline n=n_{1}+n_{2},\;P\in\Gamma(TF)$ and $U_{1}\in\Gamma(TB),\;U_{2}\in\Gamma(TF).$
If $n_{1}=1,$ then $M$ is a sample Riemannian product.
\end{Theorem}
\begin{proof}
If $n_{1}=1,$ then ${S}^{B}=0,$ by the second equation of $(34)$ we can get
$$0=a+n_{2}\frac{\Delta_{B}f}{f}+b\Rightarrow \frac{\Delta_{B}f}{f}=-\frac{a+b}{n_{2}}\triangleq c_{0}\Rightarrow
\Delta_{B}f=c_{0}f.$$
Then the Laplacian has constant sign and hence $f$ is constant.
\end{proof}
\begin{Theorem}
Let $M=B \times_{f}F$ be a generalized quasi-Einstein warped product, ${\rm dim}B=n_{1}\geqslant2,\;{\rm dim}F=n_{2}\geqslant2,\;{\rm dim}M=\overline n=n_{1}+n_{2},\;P\in\Gamma(TF)$ and $U_{1}\in\Gamma(TB),\;U_{2}\in\Gamma(TF).$
If $c\neq0,$ and $P$ is parallel on $F$ with respect to the Levi-Civita connection on $F,$ then $M$ is a sample Riemannian product.
\end{Theorem}
\begin{proof}
Consider in the second equation of $(28)$ that $V,W$ are orthogonal vector fields tangent to $F$ such that $g_{M}(V,U_{2})\neq0,\;g_{M}(W,U_{2})\neq0$ and $\pi(W)=0,$ Then $g_{F}(V,W)=0.$ Since $P$ is parallel on $F$ with respect to the Levi-Civita connection on $F,$ we have $\nabla_{V}P=0,$ so the second equation of $(28)$ becomes:
$$Ric^{F}(V,W)=cf^{4}g_{F}(V,U_{2})g_{F}(W,U_{2}).  \eqno{(46)}$$
Taking in consideration the different domains of definition of the functions that appear in the equation (46), we obtain that $f$ is constant.
\end{proof}

{\bfseries Acknowledgement.}
We would like to thank the referee for his(her) careful reading and helpful comments.


\clearpage
\end{document}